\documentclass[11pt]{article}

\usepackage[a4paper,margin=1in]{geometry}

\usepackage[T1]{fontenc}
\usepackage[utf8]{inputenc}
\usepackage[english]{babel}

\usepackage{amsmath,amssymb,amsthm}
\usepackage{mathtools}

\usepackage{url}
\usepackage[colorlinks=true,linkcolor=blue,citecolor=blue,urlcolor=blue,hypertexnames=false]{hyperref}
\usepackage[noabbrev]{cleveref}

\usepackage{graphicx}

\theoremstyle{plain}
\newtheorem{theorem}{Theorem}[section]
\newtheorem{lemma}[theorem]{Lemma}
\newtheorem{proposition}[theorem]{Proposition}
\newtheorem{corollary}[theorem]{Corollary}

\theoremstyle{definition}

\newtheorem{assumption}[theorem]{Assumption}

\theoremstyle{remark}
\newtheorem{remark}[theorem]{Remark}

\title{Threshold-Safe Shock Absorption in a Compartmental Voter-Flow Model:\
A Conservative Impulse-Control Benchmark}

\author{%
  Alexander Omelchenko\thanks{%
    Constructor University Bremen gGmbH, Campus Ring~1, 28759 Bremen,
    Germany. Email: \texttt{aomelchenko@constructor.university}.}%
}

\date{\today}

\begin{document}

\maketitle

%% =============================================================
%% Abstract
%% =============================================================

\begin{abstract}
\noindent
We formulate a deterministic threshold-safety problem for a reduced
compartmental voter-flow model.  An exogenous load enters an alienation
reservoir; between releases the reservoir recovers exponentially.  Near the
mainstream baseline the compartmental dynamics have a linear-stability
threshold \(\Delta_c\): below this level the mobilised component contracts,
while above it transient amplification is possible.  The paper introduces an
impulse-control layer for this threshold mechanism.  The threshold is obtained
from the local stability boundary of the reduced dynamical system, the exposure
functional is tied to positive logarithmic amplification, and the scalar
reservoir model is proved to be a conservative envelope of the nonlinear
voter-flow dynamics.  This bridge yields explicit safety benchmarks: the
single-release exposure and its zero buffer, the complete-relaxation splitting
problem with fixed per-release overhead, the finite-recovery constant-peak
profile, and the fixed-horizon capacity frontier \(\Delta_c(1+\rho T)\).  The
scalar recurrence used after the reduction is the familiar leaky-reservoir
skeleton also found in multiple-dose pharmacokinetics, fractionated
radiotherapy, reservoir operation, and setup-cost scheduling.  Its role here is
to make the threshold regimes of the compartmental shock-absorption model
analytically transparent.
\end{abstract}



% =================================================================
% Introduction
% =================================================================

\section{Introduction}
\label{sec:introduction}

This paper studies a deterministic threshold-safety problem for a reduced
compartmental model of political shock absorption.  A fixed nonnegative load is
delivered to a recovering reservoir through a finite sequence of nonnegative
impulses.  Between impulses the reservoir decays.  The system is treated as
locally safe while the reservoir remains below a critical level and as
threshold-active once that level is crossed.  The mathematical questions are how
to split the load, how many releases to use, and how recovery time and fixed
per-release overhead affect threshold safety.

The motivation is a stylised voter-flow model rather than a general theory of
political behaviour.  An external event---for example, a migration-related
administrative surge, an energy-price shock, a fiscal constraint, a security
crisis, or an administrative overload---is represented only by its effective
load on the system.  The load itself is not a decision variable.  The decision
variable, if present, is the absorption profile: one large implementation step,
a small number of batches, or a longer sequence of smaller releases separated
by recovery periods.  The paper asks when the same total load can be staged so
that a structural threshold is not crossed.

The immediate modelling context is the compartmental political-dynamics
literature in which radicalisation, de-radicalisation, and voter-flow mechanisms
are represented by ODE systems.  Existing models primarily analyse the ODEs
themselves: well-posedness, equilibria, reproduction numbers or spectral
thresholds, global stability, and bifurcation regimes
\cite{McCluskeySantoprete2018,SantopreteXu2018,Santoprete2019}.  In related
work, threshold voter-flow models on the probability simplex were used to
distinguish transient state shocks from durable structural threshold crossing
\cite{Omelchenko2026Threshold}.  There is also a broader
optimal-control literature for radicalisation, terrorism, and extreme-ideology
compartmental models, usually based on continuous prevention, disengagement,
deradicalisation, or governance controls and Pontryagin-type necessary
conditions \cite{AzizahBakhtiarSianturi2023,TapsobaSimporeTraore2025}.  The
present problem is different: the control is not a transition-rate intervention
or a persuasion variable, but a finite sequence of impulses describing how a
fixed external load is staged into a threshold reservoir.  To the author's
knowledge, this threshold-safe impulse-staging problem has not been formulated
for the voter-flow threshold class considered here.

The methodological route into the problem came from shock splitting in
supersonic gas dynamics.  Earlier work on optimal shock-wave systems developed
an explicitly solvable discrete-control viewpoint for replacing one large
irreversible step by a cascade of smaller ones \cite{MO1998,MO2003a,MO2006}.
A direct Rankine--Hugoniot analogue is not available here.  In gas dynamics,
shock jumps are constrained by a hyperbolic system of conservation laws, an
equation of state, and entropy admissibility
\cite{CourantFriedrichs1976,LiepmannRoshko1957,Anderson2003,
HendersonMenikoff1998}.  The present model is an ODE system; the update
\(A(t_k^+)=A(t_k^-)+q_k\) is an externally imposed control, not an internal
weak-solution discontinuity; and there is no thermodynamic closure that would
force a unique jump-loss functional.  The replacement is the balance and
comparison structure of the ODE model: a linear balance functional dissipates
along the autonomous flow and increments at impulses, the threshold
\(\Delta_c\) is a stability boundary, the exposure bounds positive logarithmic
amplification of the mobilised component, and the scalar impulse layer is a
conservative envelope of the nonlinear dynamics.  This is the ODE-level version
of the principle that the loss should be derived from the governing dynamics
rather than postulated externally.

Once this threshold and envelope have been extracted, the scalar impulse layer
has close analogues in established operations-research and applied-mathematics
settings.  The recurrence
\[
  A_k=\lambda A_{k-1}+q_k, \qquad \sum_{k=1}^n q_k=Q,
\]
is the same leaky-storage skeleton that appears in one-compartment multiple
dosing, fractionated radiotherapy with incomplete repair, reservoir operation,
and release scheduling with setup costs
\cite{RowlandTozer2011,Bauer2014,WithersThamesPeters1983,Fowler1989,Yeh1985,
WagnerWhitin1958}.  Cumulative carbon-budget models provide a policy-level
analogue of allocating a remaining load under a threshold constraint
\cite{Allen2009,Meinshausen2009}.  These parallels are used deliberately: they
provide transparent benchmark problems for the scalar envelope.  The claim of
the paper is the bridge from a compartmental voter-flow threshold to this
impulse-control benchmark, not the rediscovery of the scalar recurrence.

The paper records four consequences of this reduction.

\smallskip\noindent
\emph{First}, the exposure generated by a single release is computed in closed
form.  The exposure is exactly zero for releases not exceeding \(\Delta_c\), is
convex on the nonnegative axis, is continuously differentiable at the
threshold, and has quadratic onset just above it.  Thus the benefit of splitting
is not introduced by an externally chosen quadratic penalty; it follows from the
threshold dynamics themselves.

\smallskip\noindent
\emph{Second}, the scalar exposure is tied back to the nonlinear two-variable
compartmental system.  A linear balance functional dissipates along the full
dynamics and increments by a fixed amount at each impulse.  The resulting
single-reservoir exposure bounds the positive logarithmic amplification of the
mobilised component and is sharp in the small-mobilisation limit.  The reduced
impulse model is therefore used as a conservative envelope, not as an exact
replacement of the nonlinear dynamics.

\smallskip\noindent
\emph{Third}, in the complete-relaxation benchmark, where the reservoir returns
to baseline before the next release, the exposure-minimising split of a fixed
load into a fixed number of releases is uniform.  Full threshold safety is
achieved at the finite count \(N_{\rm safe}=\lceil Q/\Delta_c\rceil\).  If each
release carries a fixed overhead \(K\), the model gives an explicit boundary
between the regime in which full safety is cost-optimal and the regime in which
the optimum accepts positive residual exposure to avoid additional stages.

\smallskip\noindent
\emph{Fourth}, under finite recovery the paper separates peak safety from
cumulative exposure.  For a fixed number of releases, the peak-minimising
benchmark has a closed-form front-loaded solution: the first release fills the
reservoir to the target peak, and subsequent releases replenish the portion that
has decayed since the previous stage.  Under a fixed crisis horizon \(T\), the
reduced model also yields the supremal capacity frontier
\(\Delta_c(1+\rho T)\): loads above this value cannot be kept below the
threshold by any finite splitting strategy, while loads below it can be made
safe by sufficiently fine staging.

The distinction between the two scalar objectives used below is important.  The
complete-relaxation part minimises cumulative threshold exposure.  The
finite-recovery part minimises the peak reservoir level and is therefore a
safety-capacity benchmark.  These are two reductions of the same threshold
mechanism, but they are not the same optimisation problem.

The paper is organised as follows.  Section~\ref{sec:background} fixes the
absorption interpretation and modelling dictionary.  Section~\ref{sec:problem}
introduces the reduced compartmental system, the balance identity, the
conservative impulse envelope, and the discrete release recursion.
Section~\ref{sec:single_shock} derives the single-shock exposure function and
its spectral-threshold interpretation.  Section~\ref{sec:full_relaxation}
solves the complete-relaxation splitting problem and incorporates fixed
per-stage overhead.  Section~\ref{sec:finite_recovery_bellman} gives the
finite-recovery peak-safety benchmark and the fixed-horizon capacity frontier.
Section~\ref{sec:phase_diagram} summarises the resulting phase structure, and
Section~\ref{sec:discussion} discusses the conservative status of the reduced
impulse model and possible extensions.

% =================================================================
% Background and modelling assumptions
% =================================================================

\section{Background and modelling assumptions}\label{sec:background}

This section fixes the modelling dictionary used in the formal analysis.  The
political terminology below is interpretive; the equations and optimisation
problems are introduced in Section~\ref{sec:problem}.  The term ``safe'' is used
only in the technical sense of local non-amplification in the reduced threshold
system.

A political system may face an external load whose origin is outside the reduced
two-variable mechanism: a migration-related administrative surge, an
energy-price shock, a fiscal constraint, a security crisis, or an abrupt
administrative overload.  The model represents such an event by a fixed
nonnegative quantity \(Q\).  The decision variable is not the existence or
substantive content of the shock, but its absorption profile: one large release,
a small number of batches, or a longer sequence of smaller releases separated by
recovery periods.

The reduced model uses two state variables.  The reservoir \(A(t)\) is an
effective stock of stress, backlog, or alienation.  It dissipates at rate
\(\rho>0\) when no fresh load is added.  The variable \(S(t)\) is a local
mobilisation intensity near the mainstream baseline.  It is not used below as a
globally bounded population share; it is the local activity coordinate whose
linear growth or decay determines whether a transient mobilisation window is
open.  The following dictionary records the notation.

\begin{center}
\small
\begin{tabular}{@{}p{0.22\textwidth}p{0.13\textwidth}p{0.54\textwidth}@{}}
\hline
\textbf{Modelling term} & \textbf{Symbol} & \textbf{Mathematical role}\\
\hline
Total external load & \(Q\) & Fixed nonnegative quantity to be absorbed.\\[2pt]
Release or batch & \(q_j\) & Nonnegative impulse added to the reservoir; the releases satisfy
\(\sum_j q_j=Q\).\\[2pt]
Alienation reservoir & \(A(t)\) & Effective stress coordinate; it decays between releases at rate \(\rho\).\\[2pt]
Mobilisation intensity & \(S(t)\) & Local amplitude of non-mainstream mobilisation near the baseline; its
linear growth is controlled by \(A(t)\).\\[2pt]
Critical threshold & \(\Delta_c\) & Reservoir level at which the baseline changes local stability.\\[2pt]
Recovery rate & \(\rho\) & Rate at which the reservoir dissipates between releases.\\[2pt]
Per-release overhead & \(K\) & Fixed administrative or operational cost of using one additional stage.\\
\hline
\end{tabular}
\end{center}

Three dimensionless quantities organise the scalar benchmarks: \(Q/\Delta_c\),
which measures the load in threshold units; \(\rho T\), which measures available
recovery over a horizon \(T\); and \(K\rho/(\mu-\beta)\), which measures fixed
per-release overhead in the natural one-shock exposure scale.  The model is not
an electoral forecast, a persuasion model, or a statistical description of a
particular country.  It isolates a threshold-safe absorption mechanism for a
fixed load in a conservative scalar envelope.

% ============================================================
% Problem formulation: threshold-safe shock absorption
% ============================================================

\section{A reduced shock-absorption problem}\label{sec:problem}

We now turn the modelling dictionary into equations.  The section has two
roles.  First, it introduces the reduced two-variable threshold mechanism and
identifies the local stability threshold used later.  Second, it replaces the
reservoir equation by a scalar impulse envelope and proves that this envelope is
conservative for the nonlinear system.

\subsection{Origin of the symmetric two-variable reduction}
\label{subsec:origin_symmetric_reduction}

The reduced system used below is a self-contained symmetric local block
motivated by the author's voter-flow threshold models
\cite{Omelchenko2026Threshold}.  The short derivation is included here so that
the control problem does not rely on unpublished model details.  Let \(C\) denote the
mainstream baseline level, \(S\) the local mobilisation intensity, and \(A\) the
alienation reservoir.  In the symmetric one-polarity reduction, the mobilisation
coordinate has the linear-in-\(S\) balance
\begin{equation}\label{eq:S_origin_balance}
  \dot S=S\bigl(\beta C+\delta A-\mu\bigr),
\end{equation}
where \(\beta C\) is baseline mobilisation from the mainstream compartment,
\(\delta A\) is reactivation from the reservoir, and \(\mu\) is return or
reabsorption.  The reservoir recovers autonomously and is also depleted by the
mobilisation channel:
\begin{equation}\label{eq:A_origin_balance}
  \dot A=-\rho A-\delta SA.
\end{equation}
In the local symmetric coordinates, \(C=1-A-S\).  Substituting this expression
into \eqref{eq:S_origin_balance} gives
\[
  \dot S
  =S\bigl[(\beta-\mu)-\beta S+(\delta-\beta)A\bigr],
\]
while \eqref{eq:A_origin_balance} gives \(\dot A=-(\rho+\delta S)A\).  Thus the
closed two-variable reduction is
\begin{equation}\label{eq:reduced_SA_full}
\begin{aligned}
  \dot S
  &=S\bigl[(\beta-\mu)-\beta S+(\delta-\beta)A\bigr],\\
  \dot A
  &=-(\delta S+\rho)A.
\end{aligned}
\end{equation}
The subsequent analysis uses only the following structural properties of this
local block: the nonnegative invariance of \(S\), an affine near-baseline growth
pressure increasing in \(A\), autonomous reservoir recovery, additional
reservoir depletion when \(S>0\), and additive impulses in the reservoir
coordinate.  The symmetric voter-flow reduction above is one source of this
structure; the scalar benchmarks below apply to this structural class.

\begin{assumption}[Affine stress-coordinate continuation]
\label{ass:affine_continuation}
The compartmental reduction fixes the local stability threshold and the sign of
the near-baseline growth pressure.  In the impulse layer, \(A\) is not used as a
literal population share once a trajectory leaves the local simplex range.  It
is used as an effective stress coordinate calibrated so that
\(A=\Delta_c\) is the local stability boundary.  For \(A>\Delta_c\), the affine
continuation
\[
  g(A)=(\delta-\beta)(A-\Delta_c)
\]
is used as a conservative risk proxy, not as a claim that the original
population-share interpretation remains valid globally.
\end{assumption}

\subsection{The symmetric threshold mechanism}
\label{subsec:symmetric_threshold_mechanism}

All parameters in \eqref{eq:reduced_SA_full} are positive.  The state \(S=0\)
is invariant, and if \(S(0)\ge0\) then \(S(t)\ge0\) as long as the solution
exists.  The term \(-\beta S^2\) is a local saturation term.  It should not be
read as a full population-simplex constraint; it only prevents the local
mobilisation coordinate from growing linearly without bound inside the reduced
model.  The product term \(-\delta SA\) in the second equation is important for
the conservative-envelope argument: the full two-variable system drains the
reservoir at least as fast as the scalar envelope introduced later.

We work in the shock-sensitive regime
\begin{equation}\label{eq:shock_sensitive_regime}
  0<\beta<\mu<\delta.
\end{equation}
If the reservoir is empty, then the linear growth rate of \(S\) at \(S=0\) is
\(\beta-\mu<0\), so the mainstream baseline is locally attracting.  If the
reservoir is sufficiently large, the term \((\delta-\beta)A\) can reverse this
sign.  The near-baseline growth pressure is
\begin{equation}\label{eq:gA}
  g(A):=(\beta-\mu)+(\delta-\beta)A.
\end{equation}
It changes sign at
\begin{equation}\label{eq:Deltac_problem}
  \Delta_c:=\frac{\mu-\beta}{\delta-\beta}.
\end{equation}
Under \eqref{eq:shock_sensitive_regime}, one has \(0<\Delta_c<1\).  This
inequality reflects the normalisation of the symmetric compartmental block; it
is not an upper bound on admissible impulse loads.  When \(A\le\Delta_c\), the
linearised mobilisation intensity does not grow.  When \(A>\Delta_c\), the
reduced model enters a threshold-active interval.

This motivates the path-level threshold exposure
\begin{equation}\label{eq:threshold_exposure_path}
  \mathcal E[A]
  :=\int_0^\infty [g(A(t))]_+\,dt,
  \qquad [x]_+:=\max\{x,0\}.
\end{equation}
For small \(S\), the logarithmic growth rate of \(S\) is \(g(A)+O(S)\).  Hence
\eqref{eq:threshold_exposure_path} records the positive part of the
near-mainstream growth pressure generated by the reservoir path.  Positive-part
exposure functionals occur in threshold and tolerance models
\cite{RowlandTozer2011,Fowler1989}; here the threshold and pressure come from
linearising \eqref{eq:reduced_SA_full}.

\begin{proposition}[Balance identity and logarithmic exposure]
\label{prop:balance_log_exposure}
Along the full symmetric dynamics \eqref{eq:reduced_SA_full}, without impulses,
set
\begin{equation}\label{eq:alpha_gamma_balance_defs}
  \alpha:=\delta-\beta,
  \qquad
  \gamma:=\mu-\beta,
  \qquad
  \Phi(S,A):=S+\frac{\alpha}{\delta}A .
\end{equation}
Then, on the nonnegative quadrant,
\begin{equation}\label{eq:balance_identity}
  \frac{d}{dt}\Phi(S(t),A(t))
  =-\gamma S(t)-\beta S(t)^2-\frac{\alpha\rho}{\delta}A(t)
  \qquad
  \text{between impulses.}
\end{equation}
In particular, \(\Phi\) is a strict balance functional away from the origin.  If
an impulse of size \(q\ge0\) is applied to the reservoir, then
\begin{equation}\label{eq:balance_jump}
  \Phi(S,A+q)-\Phi(S,A)=\frac{\alpha}{\delta}q .
\end{equation}
Moreover, for every solution with \(S(0)>0\) and every \(T>0\),
\begin{equation}\label{eq:log_growth_bound}
  \log \frac{S(T)}{S(0)}
  \le
  \int_0^T [g(A(t))]_+\,dt .
\end{equation}
\end{proposition}

\begin{proof}
The derivative of \(\Phi\) along \eqref{eq:reduced_SA_full} is
\begin{align*}
  \frac{d}{dt}\Phi(S,A)
  &=S\bigl[-\gamma-\beta S+\alpha A\bigr]
    +\frac{\alpha}{\delta}\bigl[-(\rho+\delta S)A\bigr] \\
  &=-\gamma S-\beta S^2-\frac{\alpha\rho}{\delta}A,
\end{align*}
which gives \eqref{eq:balance_identity}.  The jump identity
\eqref{eq:balance_jump} follows from the definition of \(\Phi\).  Finally, for
\(S(t)>0\),
\begin{equation*}
  \frac{d}{dt}\log S(t)=g(A(t))-\beta S(t)
  \le g(A(t))
  \le [g(A(t))]_+ .
\end{equation*}
Integrating over \([0,T]\) gives \eqref{eq:log_growth_bound}.
\end{proof}

This proposition is the first genuinely model-specific step.  The impulse
\(q\) contributes a fixed amount of the balance potential \(\Phi\), but the
risk of transient mobilisation is not controlled by this total input alone.  It
is controlled by the portion of the trajectory on which the logarithmic growth
pressure of \(S\) becomes positive.  Thus \(\mathcal E[A]\) is used as a
conservative upper bound on positive logarithmic amplification of the mobilised
component, not as a quadratic penalty chosen for convenience.

\subsection{The scalar impulse envelope}
\label{subsec:impulse_layer}

The control problem concerns the way a fixed external load enters the reservoir.
The load itself is fixed; the decision is its decomposition into effective
releases.  By analogy with standard linear-storage and multiple-dosing models,
we use the mobilisation-free reservoir equation
\begin{equation}\label{eq:A_impulse_continuous}
  \dot A(t)=-\rho A(t)
  \qquad
  \text{for } t\notin\{t_1,\ldots,t_n\},
\end{equation}
with jumps
\begin{equation}\label{eq:A_impulse_jump}
  A(t_k^+)=A(t_k^-)+q_k,
  \qquad q_k\ge0,
  \qquad k=1,\ldots,n.
\end{equation}
The releases satisfy the fixed budget constraint
\begin{equation}\label{eq:shock_budget}
  \sum_{k=1}^n q_k=Q,
  \qquad Q>0.
\end{equation}
The impulse idealisation means that each release occurs on a timescale much
shorter than the recovery time of the reservoir.

This scalar layer is the canonical leaky-reservoir equation used in dosing,
storage, and release-scheduling models
\cite{RowlandTozer2011,Bauer2014,Yeh1985,WithersThamesPeters1983,Fowler1989}.
Its role here is to provide a conservative envelope for the reservoir component
of \eqref{eq:reduced_SA_full}.  The next proposition makes this precise.

\subsection{Conservative status of the reduced impulse layer}
\label{subsec:conservative_impulse_layer}

\begin{proposition}[Conservative envelope property]
\label{prop:conservative_envelope}
Consider the same nonnegative impulse sequence \(q_1,\ldots,q_n\) applied at the
same times to the scalar envelope and to the full symmetric system.  Let
\(A_{\rm red}\) solve
\begin{equation}\label{eq:A_red_envelope}
  \dot A_{\rm red}=-\rho A_{\rm red},
  \qquad
  A_{\rm red}(t_j^+)=A_{\rm red}(t_j^-)+q_j,
\end{equation}
and let \((S_{\rm full},A_{\rm full})\) solve
\begin{equation}\label{eq:full_envelope_system}
\begin{aligned}
  \dot S_{\rm full}
  &=S_{\rm full}\bigl[(\beta-\mu)-\beta S_{\rm full}
      +(\delta-\beta)A_{\rm full}\bigr],\\
  \dot A_{\rm full}
  &=-(\rho+\delta S_{\rm full})A_{\rm full},
\end{aligned}
\end{equation}
with the same impulse updates for \(A_{\rm full}\).  Assume that the initial
reservoir levels are equal and nonnegative, and that \(S_{\rm full}(0)\ge0\).
Then
\begin{equation}\label{eq:A_full_leq_red}
  A_{\rm full}(t)\le A_{\rm red}(t)
  \qquad\text{for all }t\ge0.
\end{equation}
Consequently, for every \(T>0\),
\begin{equation}\label{eq:exposure_reduced_dominates_full}
  \int_0^T [g(A_{\rm full}(t))]_+\,dt
  \le
  \int_0^T [g(A_{\rm red}(t))]_+\,dt.
\end{equation}
In particular, if \(A_{\rm red}(t)\le\Delta_c\) for all \(t\in[0,T]\), then the
near-mainstream growth pressure in the full symmetric system is nonpositive on
that interval.
\end{proposition}

\begin{proof}
The equation for \(S_{\rm full}\) preserves nonnegativity, so
\(S_{\rm full}(t)\ge0\).  Between two consecutive impulses, the full reservoir
satisfies
\[
  \dot A_{\rm full}=-(\rho+\delta S_{\rm full})A_{\rm full},
\]
whereas the scalar envelope satisfies \(\dot A_{\rm red}=-\rho A_{\rm red}\).
Thus the full reservoir decays at least as fast as the scalar envelope.  At an
impulse time both reservoir variables receive the same nonnegative jump, which
preserves the ordering.  The comparison principle therefore gives
\eqref{eq:A_full_leq_red} on each inter-release interval and hence on the whole
time line.

Because \(\delta>\beta\), the function \(g(A)\) in \eqref{eq:gA} is increasing.
Thus \([g(A_{\rm full}(t))]_+\le [g(A_{\rm red}(t))]_+\), and integration gives
\eqref{eq:exposure_reduced_dominates_full}.  If
\(A_{\rm red}(t)\le\Delta_c\), then \(A_{\rm full}(t)\le\Delta_c\), so
\(g(A_{\rm full}(t))\le0\) by the definition of \(\Delta_c\).
\end{proof}

The converse is not asserted.  If the scalar envelope crosses the threshold,
the full nonlinear system may still be moderated by the additional depletion
term \(-\delta SA\) and by the saturation term \(-\beta S^2\).  Thus the scalar
layer is a conservative design problem: safety in the envelope is a sufficient
condition for local threshold safety in the symmetric mechanism, while failure
of the envelope condition is not a prediction of unavoidable mobilisation.

Figure~\ref{fig:envelope_comparison} illustrates the comparison for one
representative impulse sequence.  The parameters are \(\beta=0.6\), \(\mu=1\),
\(\delta=1.8\), \(\rho=0.5\), so \(\Delta_c=1/3\); impulses are applied at
\(t=0,2,4,6,8\) with sizes \(0.46,0.24,0.24,0.24,0.24\), and \(S(0)=0.08\).
The scalar envelope crosses the threshold in this illustrative run; the purpose
of the plot is to show dominance of the full reservoir by the conservative
scalar benchmark, not to display a safe schedule.

\begin{figure}[t]
  \centering
  \includegraphics[width=0.72\textwidth]{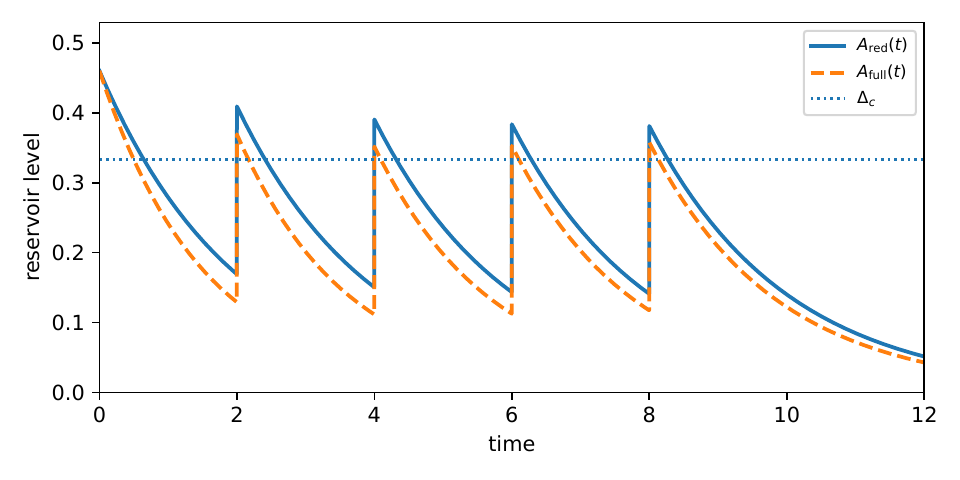}
  \caption{Scalar envelope and nonlinear reservoir path}
  \label{fig:envelope_comparison}
\end{figure}

\subsection{Discrete releases and time conventions}
\label{subsec:discrete_releases_time_conventions}

The continuous impulse notation fixes the model, but the later optimisation
variables are finite-dimensional.  Suppose first that consecutive releases are
separated by a fixed time interval \(\tau>0\).  The reservoir then decays by the
factor
\begin{equation}\label{eq:lambda_def}
  \lambda:=e^{-\rho\tau}\in(0,1).
\end{equation}
Let \(A_k\) be the reservoir level immediately after the \(k\)-th release,
starting from \(A_0=0\).  The scalar envelope becomes
\begin{equation}\label{eq:Ak_discrete}
  A_k=\lambda A_{k-1}+q_k,
  \qquad k=1,\ldots,n.
\end{equation}
The number \(1-\lambda\) is the fraction of the reservoir that dissipates
between successive releases.  In pharmacokinetic language, \(\lambda\) is the
inter-dose carry-over factor and \(A_k\) is the post-dose level; in storage
language, \(\lambda\) is the retention factor.  Below we use the same recurrence
only as a threshold-safety envelope for the compartmental model.

Two time conventions will be used.  In the complete-relaxation benchmark, the
intervals between releases are long enough that each release starts from an
empty reservoir.  This leads to the static cumulative-exposure splitting
problem in Section~\ref{sec:full_relaxation}.  In the finite-recovery benchmark,
\(\lambda\) is fixed and the recursion \eqref{eq:Ak_discrete} is used directly;
this leads to the peak-safety problem in
Section~\ref{sec:finite_recovery_bellman}.

A different convention is needed when the whole absorption process must be
completed within a fixed crisis horizon \(T\).  If \(n\ge2\) releases are equally
spaced over that horizon, then
\begin{equation}\label{eq:lambda_n_fixed_T}
  \tau_n:=\frac{T}{n-1},
  \qquad
  \lambda_n:=e^{-\rho T/(n-1)}.
\end{equation}
Increasing \(n\) then gives more releases but less recovery time between two
neighbouring releases.  This trade-off is responsible for the fixed-horizon
capacity bound derived in Section~\ref{sec:finite_recovery_bellman}.

The dimensionless parameters that organise the reduced problem are
\begin{equation}\label{eq:dimensionless_parameters}
  \frac{Q}{\Delta_c},
  \qquad
  \rho T,
  \qquad
  \frac{K\rho}{\mu-\beta}.
\end{equation}
They measure, respectively, the load size relative to the threshold, the
available recovery capacity over the horizon, and the fixed per-release
overhead relative to the natural one-shock exposure scale
\((\mu-\beta)/\rho\).  These quantities reappear in the phase diagram in
Section~\ref{sec:phase_diagram}.

The next section analyses the model-specific ingredient of the scalar layer:
the threshold exposure produced by a single release when the threshold and the
growth pressure are derived from the compartmental mechanism above.

% ============================================================
% Central block: single-shock threshold exposure
% ============================================================

\section{The single-shock threshold exposure}\label{sec:single_shock}

This section computes the one-release exposure associated with the stability
pressure
\[
  g(A)=(\beta-\mu)+(\delta-\beta)A
\]
and the threshold \(\Delta_c\).  Analytically, this is the integrated exceedance
of an exponentially decaying stock above a threshold, a calculation that appears
in dosing, tolerance, and storage models
\cite{RowlandTozer2011,Bauer2014,Fowler1989}.  In the present model the
threshold is the linear stability boundary of the voter-flow mechanism.

A release of size \(q\ge0\) is added to an initially empty reservoir at time
\(t=0\).  No further load is added.  The scalar envelope is therefore
\begin{equation}\label{eq:A_single_shock_path}
  A_q(t)=q e^{-\rho t},
  \qquad t\ge0.
\end{equation}
Keeping the notation
\[
  \alpha:=\delta-\beta>0,
  \qquad
  \gamma:=\mu-\beta>0,
  \qquad
  \Delta_c=\frac{\gamma}{\alpha},
\]
we have
\begin{equation}\label{eq:g_alpha_form}
  g(A)=\alpha(A-\Delta_c).
\end{equation}
The single-release exposure is
\begin{equation}\label{eq:E_single_def}
  E(q):=\int_0^\infty [g(qe^{-\rho t})]_+\,dt
       =\int_0^\infty
       \left[(\beta-\mu)+(\delta-\beta)q e^{-\rho t}\right]_+dt,
\end{equation}
where \([x]_+:=\max\{x,0\}\).  It records only the time interval during which
the reservoir path lies above the stability threshold.

\begin{lemma}[Single-release threshold exposure]\label{lem:single_shock_exposure}
Assume \(0<\beta<\mu<\delta\).  The function
\(E:[0,\infty)\to[0,\infty)\) defined by \eqref{eq:E_single_def} is finite and
is given by
\begin{equation}\label{eq:E_piecewise_formula}
  E(q)=
  \begin{cases}
    0,
    &0\le q\le \Delta_c,\\[0.6em]
    \displaystyle
    \frac{\delta-\beta}{\rho}
    \left(
      q-\Delta_c-\Delta_c\log\frac{q}{\Delta_c}
    \right),
    &q>\Delta_c.
  \end{cases}
\end{equation}
Moreover, \(E\) is nonnegative, nondecreasing, convex on \([0,\infty)\),
identically zero on \([0,\Delta_c]\), and strictly convex on
\((\Delta_c,\infty)\).  It is continuously differentiable, with
\begin{equation}\label{eq:E_derivatives}
  E'(q)=0 \quad (0<q<\Delta_c),
  \qquad
  E'(q)=\frac{\delta-\beta}{\rho}
        \left(1-\frac{\Delta_c}{q}\right)
        \quad (q>\Delta_c),
\end{equation}
and, for \(q>\Delta_c\),
\begin{equation}\label{eq:E_second_derivative}
  E''(q)=\frac{\delta-\beta}{\rho}\frac{\Delta_c}{q^2}>0.
\end{equation}
Near the threshold,
\begin{equation}\label{eq:E_quadratic_asymptotics}
  E\bigl(\Delta_c(1+\varepsilon)\bigr)
  =
  \frac{(\delta-\beta)\Delta_c}{2\rho}\,\varepsilon^2
  +O(\varepsilon^3),
  \qquad \varepsilon\downarrow0.
\end{equation}
At large release sizes,
\begin{equation}\label{eq:E_large_q_asymptotics}
  E(q)
  =
  \frac{\delta-\beta}{\rho}\,q
  -\frac{\mu-\beta}{\rho}\log\frac{q}{\Delta_c}
  -\frac{\mu-\beta}{\rho},
  \qquad q>\Delta_c.
\end{equation}
\end{lemma}

\begin{proof}
If \(0\le q\le\Delta_c\), then \(qe^{-\rho t}\le\Delta_c\) for all
\(t\ge0\), so the positive part in \eqref{eq:E_single_def} vanishes.  If
\(q>\Delta_c\), the active interval is
\[
  0\le t<t_q,
  \qquad
  t_q=\frac1\rho\log\frac{q}{\Delta_c}.
\]
Hence
\[
  E(q)
  =\alpha\int_0^{t_q}\bigl(qe^{-\rho t}-\Delta_c\bigr)\,dt
  =\frac{\alpha}{\rho}
    \int_{\Delta_c}^{q}\frac{u-\Delta_c}{u}\,du,
\]
which gives \eqref{eq:E_piecewise_formula}.  Differentiating the explicit
formula gives \eqref{eq:E_derivatives} and \eqref{eq:E_second_derivative}; the
right derivative at \(q=\Delta_c\) is zero, so \(E\) is \(C^1\) across the
threshold.  The stated monotonicity and convexity follow from these derivative
formulae.  Finally, \eqref{eq:E_quadratic_asymptotics} follows from
\(\varepsilon-\log(1+\varepsilon)=\varepsilon^2/2+O(\varepsilon^3)\), and
\eqref{eq:E_large_q_asymptotics} follows from \(\alpha\Delta_c=\mu-\beta\).
\end{proof}

The qualitative information carried by Lemma~\ref{lem:single_shock_exposure} is
the zero buffer and the convex superthreshold branch.  The one-release exposure
is exactly zero below \(\Delta_c\).  Above the threshold it is convex, with
quadratic onset.  Therefore the benefit of splitting a release
into smaller releases is not imposed by an external quadratic penalty.  It is a
consequence of the threshold pressure generated by the compartmental dynamics.

\subsection{Spectral origin of the threshold}
\label{subsec:spectral_threshold_bridge}

Formula \eqref{eq:E_piecewise_formula} becomes a voter-flow quantity because
the threshold is the point at which the linearised mobilisation block crosses
its normalised spectral threshold.

In the mobilisation-free symmetric compartmental state, before the affine
continuation used by the impulse envelope, let \(A\in[0,1]\) be the reservoir
level and let \(C=1-A\) be the baseline level.  Define
\begin{equation}\label{eq:R_A_def}
  R(A):=\frac{\beta C+\delta A}{\mu}
       =\frac{(1-A)\beta+\delta A}{\mu}.
\end{equation}
Here \(R(A)\) is the normalised linear growth factor of the mobilised component
when the reservoir level is frozen at \(A\).  In a higher-dimensional
compartmental model, the analogous object would be a Perron--Frobenius spectral
radius.

\begin{proposition}[Normalised threshold factor]\label{prop:spectral_threshold_bridge}
For \(R(A)\) defined by \eqref{eq:R_A_def},
\begin{equation}\label{eq:g_spectral_identity}
  g(A)=\mu\bigl(R(A)-1\bigr).
\end{equation}
Consequently,
\begin{equation}\label{eq:R_Deltac_identity}
  R(\Delta_c)=1,
\end{equation}
and \(R(A)<1\) for \(A<\Delta_c\), while \(R(A)>1\) for
\(A>\Delta_c\).  With the affine continuation of \(R\) beyond \([0,1]\), the
single-release exposure can equivalently be written as
\begin{equation}\label{eq:E_spectral_excess}
  E(q)=\mu\int_0^\infty [R(qe^{-\rho t})-1]_+\,dt.
\end{equation}
\end{proposition}

\begin{proof}
Since \(C=1-A\),
\[
  \mu\bigl(R(A)-1\bigr)
  =(1-A)\beta+\delta A-\mu
  =\beta-\mu+(\delta-\beta)A
  =g(A).
\]
Thus \(R(\Delta_c)=1\) because \(g(\Delta_c)=0\).  The function \(R\) is
strictly increasing since \(\delta>\beta\), so the crossing is unique.  Formula
\eqref{eq:E_spectral_excess} follows by substituting
\eqref{eq:g_spectral_identity} into \eqref{eq:E_single_def}.
\end{proof}

In pharmacokinetic, radiotherapy, or reservoir analogues, the relevant
tolerance level is typically specified by the applied problem.  Here it is tied
to a stability boundary of the reduced compartmental dynamics.  The scalar
exposure is therefore used as an upper measure of positive near-baseline
logarithmic amplification.

% ============================================================
% Shock splitting under complete relaxation
% ============================================================

\section{Shock splitting under complete relaxation}
\label{sec:full_relaxation}

We first consider the zero-memory limit of the scalar envelope.  After each
release the reservoir is assumed to recover to zero before the next release is
applied.  This is the full-washout or full-repair idealisation used in dosing,
fractionation, and storage models
\cite{RowlandTozer2011,Bauer2014,WithersThamesPeters1983,Fowler1989,Yeh1985}.
The exposure contributions are independent, and a split \(q_1,\ldots,q_n\) of a
fixed load \(Q\) has total exposure \(\sum_{j=1}^n E(q_j)\).  The allocation is
convex; the distinctive feature is that \(E\) has an exact zero plateau because
\(\Delta_c\) is a stability threshold.

For a fixed number \(n\ge1\) of releases and a total load \(Q>0\), define
\begin{equation}\label{eq:L_n_def}
  \mathcal L_n(Q)
  :=
  \min\left\{
      \sum_{j=1}^n E(q_j):
      q_j\ge0,\quad \sum_{j=1}^n q_j=Q
  \right\}.
\end{equation}
This is the smallest cumulative threshold exposure attainable with exactly
\(n\) completely separated releases.

\begin{theorem}[Complete-relaxation allocation]
\label{thm:complete_relaxation_splitting}
Assume \(0<\beta<\mu<\delta\), and let
\[
  \Delta_c=\frac{\mu-\beta}{\delta-\beta}.
\]
For every \(Q>0\) and every integer \(n\ge1\),
\begin{equation}\label{eq:L_n_equal_split}
  \mathcal L_n(Q)=nE\left(\frac{Q}{n}\right).
\end{equation}
Equivalently,
\begin{equation}\label{eq:L_n_piecewise}
  \mathcal L_n(Q)
  =
  \begin{cases}
    0,
    &0<Q\le n\Delta_c,\\[0.8em]
    \displaystyle
    \frac{\delta-\beta}{\rho}
    \left(
      Q-n\Delta_c
      -n\Delta_c\log\frac{Q}{n\Delta_c}
    \right),
    &Q>n\Delta_c.
  \end{cases}
\end{equation}
The minimisers are as follows.
\begin{enumerate}
  \item If \(0<Q<n\Delta_c\), every decomposition satisfying
  \begin{equation}\label{eq:safe_split_set}
    0\le q_j\le\Delta_c,
    \qquad
    \sum_{j=1}^n q_j=Q,
  \end{equation}
  is optimal, and the minimum exposure is zero.
  \item If \(Q\ge n\Delta_c\), the minimiser is unique and is the equal split
  \begin{equation}\label{eq:equal_split}
    q_1=\cdots=q_n=\frac{Q}{n}.
  \end{equation}
\end{enumerate}
Moreover, \(\mathcal L_n(Q)\) is nonincreasing in \(n\), and
\begin{equation}\label{eq:L_n_monotone}
  \mathcal L_{n+1}(Q)<\mathcal L_n(Q)
  \quad\text{if and only if}\quad
  Q>n\Delta_c.
\end{equation}
\end{theorem}

\begin{proof}
For fixed \(n\), Jensen's inequality and the convexity of \(E\) give
\[
  \sum_{j=1}^n E(q_j)
  \ge
  nE\left(\frac1n\sum_{j=1}^n q_j\right)
  =nE\left(\frac{Q}{n}\right),
\]
and equality is attained by the equal split.  This proves
\eqref{eq:L_n_equal_split}; substituting the formula from
Lemma~\ref{lem:single_shock_exposure} gives \eqref{eq:L_n_piecewise}.

If \(Q<n\Delta_c\), the objective is zero exactly for decompositions with all
\(q_j\le\Delta_c\), which gives the first minimiser statement.  If
\(Q=n\Delta_c\), zero exposure forces all \(q_j=\Delta_c\).  If
\(Q>n\Delta_c\), the average lies above the threshold.  Since \(E\) is flat only
on \([0,\Delta_c]\) and strictly convex above \(\Delta_c\), equality in
Jensen's inequality is then possible only when all \(q_j\) are equal.

Finally, for any convex function \(f\) with \(f(0)=0\), the standard scaling
inequality \(f(\theta x)\le \theta f(x)\), \(0\le\theta\le1\), implies that
\(n f(Q/n)\) is nonincreasing in \(n\).  Applying this to \(E\) gives
monotonicity of \(\mathcal L_n(Q)\).  The inequality is strict precisely when
\(E\) is not affine on \([0,Q/n]\), which here is equivalent to
\(Q/n>\Delta_c\).
\end{proof}

The finite zero-loss target is the feature that distinguishes this benchmark
from models in which every positive impulse carries positive cost.  In the
present model, splitting is useful not merely because the loss is convex, but
because releases below the stability threshold have no cumulative exposure at
all.

\begin{corollary}[Minimal number of threshold-safe releases]
\label{cor:minimal_safe_number}
Under complete relaxation, a total load \(Q>0\) can be decomposed into \(n\)
zero-exposure releases if and only if
\begin{equation}\label{eq:safe_n_condition}
  n\ge \frac{Q}{\Delta_c}.
\end{equation}
Consequently, the minimal threshold-safe number of releases is
\begin{equation}\label{eq:N_safe_def}
  N_{\rm safe}(Q)
  :=\left\lceil \frac{Q}{\Delta_c}\right\rceil .
\end{equation}
If \(Q/\Delta_c\) is not an integer, minimal safe decompositions need not be
unique.  If \(Q/\Delta_c\) is an integer, the minimal safe decomposition is
unique and consists of equal releases of size \(\Delta_c\).
\end{corollary}

\begin{proof}
By Theorem~\ref{thm:complete_relaxation_splitting}, zero exposure with \(n\)
releases is possible exactly when \(Q\le n\Delta_c\).  The uniqueness statements
follow from the minimiser description in the theorem.
\end{proof}

\subsection{Fixed overhead in the complete-relaxation limit}
\label{subsec:administrative_overhead_full_relaxation}

The previous result ignores the cost of using more stages.  We now add a fixed
overhead \(K>0\) per release.  This is the same structural device as a setup
cost in dynamic lot-sizing and related release-scheduling models
\cite{WagnerWhitin1958}: a fixed per-stage charge is balanced against the
benefit of splitting the load more finely.  The distinction is the
stage-dependent term.  Here it is the threshold exposure \(\mathcal L_n(Q)\),
not an inventory holding, shortage, or production cost.

For complete relaxation, the cost of using \(n\) releases is
\begin{equation}\label{eq:J_n_full_relaxation_def}
  \mathcal J_n^{\infty}(Q;K)
  :=nK+\mathcal L_n(Q),
  \qquad n=1,2,\ldots .
\end{equation}
The superscript \(\infty\) indicates complete recovery between releases.  The
outer problem is the integer minimisation
\begin{equation}\label{eq:J_full_relaxation_outer}
  \mathcal J^{\infty,*}(Q;K)
  :=\min_{n\ge1}\mathcal J_n^{\infty}(Q;K).
\end{equation}
This integer trade-off gives an explicit boundary between cost-optimal safety
and cost-optimal residual exposure.

Use the dimensionless variables
\begin{equation}\label{eq:r_k_defs}
  r:=\frac{Q}{\Delta_c},
  \qquad
  k:=\frac{K\rho}{\mu-\beta},
  \qquad
  N_{\rm safe}:=\lceil r\rceil .
\end{equation}
Thus \(r\) is the load measured in threshold units, and \(k\) is the overhead
measured in the one-shock exposure scale \((\mu-\beta)/\rho\).

\begin{proposition}[Overhead--exposure trade-off]
\label{prop:cost_exposure_tradeoff}
Every minimiser of \eqref{eq:J_full_relaxation_outer} belongs to the finite set
\begin{equation}\label{eq:n_candidate_finite_set}
  \{1,2,\ldots,N_{\rm safe}\}.
\end{equation}
For \(1\le n<N_{\rm safe}\), the dimensionless objective is
\begin{equation}\label{eq:Jhat_unsafe}
  \widehat{\mathcal J}_n
  :=\frac{\rho}{\mu-\beta}\mathcal J_n^{\infty}(Q;K)
  =nk+r-n-n\log\frac{r}{n},
\end{equation}
whereas the threshold-safe candidate has value
\begin{equation}\label{eq:Jhat_safe}
  \widehat{\mathcal J}_{N_{\rm safe}}
  =N_{\rm safe}k.
\end{equation}
Consequently, the threshold-safe number of releases is optimal if and only if
\begin{equation}\label{eq:k_safe_condition}
  k\le k_{\rm safe}(r),
\end{equation}
where, for \(N_{\rm safe}\ge2\),
\begin{equation}\label{eq:k_safe_def}
  k_{\rm safe}(r)
  :=
  \min_{1\le m<N_{\rm safe}}
  \frac{r-m-m\log(r/m)}{N_{\rm safe}-m},
\end{equation}
and for \(N_{\rm safe}=1\) we set \(k_{\rm safe}(r):=+\infty\).  At equality in
\eqref{eq:k_safe_condition}, the safe candidate and at least one unsafe
candidate are both optimal.

If the integer variable is relaxed to a continuous variable on the unsafe
interval \(0<n<r\), then
\begin{equation}\label{eq:Jhat_cont_def}
  \widehat{\mathcal J}(n)=nk+r-n-n\log\frac{r}{n}
\end{equation}
is strictly convex and has the unique stationary point
\begin{equation}\label{eq:n_cont_star}
  n_{\rm cont}^*=r e^{-k}
  =\frac{Q}{\Delta_c}
  \exp\left(-\frac{K\rho}{\mu-\beta}\right).
\end{equation}
\end{proposition}

\begin{proof}
For every \(n\ge N_{\rm safe}\), the exposure term is zero, so
\(\mathcal J_n^{\infty}(Q;K)=nK\).  Among these safe candidates the smallest
value is attained at \(n=N_{\rm safe}\), proving
\eqref{eq:n_candidate_finite_set}.

For \(1\le n<N_{\rm safe}\), one has \(n<r\).  Using
\eqref{eq:L_n_piecewise} and \((\mu-\beta)=(\delta-\beta)\Delta_c\),
\[
  \frac{\rho}{\mu-\beta}\mathcal L_n(Q)
  =r-n-n\log\frac{r}{n}.
\]
Adding \(nk\) gives \eqref{eq:Jhat_unsafe}, while
\eqref{eq:Jhat_safe} follows from zero exposure at \(N_{\rm safe}\).

The safe candidate is optimal exactly when
\[
  N_{\rm safe}k
  \le
  mk+r-m-m\log\frac{r}{m}
  \qquad (1\le m<N_{\rm safe}),
\]
which is equivalent to \eqref{eq:k_safe_condition}--\eqref{eq:k_safe_def}.
Finally,
\[
  \widehat{\mathcal J}'(n)=k-\log\frac{r}{n},
  \qquad
  \widehat{\mathcal J}''(n)=\frac1n>0,
\]
so the relaxed unsafe objective is strictly convex and its stationary point is
\eqref{eq:n_cont_star}.
\end{proof}

The interpretation is direct.  Without overhead, the best complete-relaxation
strategy keeps adding releases until exposure is eliminated.  With overhead,
full threshold safety may be feasible but not cost-optimal.  The curve
\(k=k_{\rm safe}(r)\) separates the two regimes: below it, the cost-optimal
solution is threshold-safe; above it, the cost-optimal solution uses fewer
releases and accepts positive residual threshold exposure.  The relaxed formula
\(n_{\rm cont}^*=r e^{-k}\) is not an integer solution formula.  It only records
the expected setup-cost effect: higher overhead pulls the optimum away from the
threshold-safe count \(r=Q/\Delta_c\).

The next section removes the complete-relaxation assumption.  Then releases are
no longer independent, and the equal split ceases to be the natural peak-safety
benchmark.

% ============================================================
% Finite recovery: peak safety and capacity
% ============================================================

\section{Finite recovery and peak-safety capacity}
\label{sec:finite_recovery_bellman}

The complete-relaxation benchmark treats consecutive releases as independent.
We now keep the scalar envelope but allow residual reservoir memory between
stages.  The same linear carry-over equation appears in multiple-dose
pharmacokinetics, incomplete-repair fractionation, and leaky-storage models
\cite{RowlandTozer2011,Bauer2014,WithersThamesPeters1983,Fowler1989,Yeh1985}.
Here it gives the peak-safety capacity benchmark associated with the
stability-derived threshold.

The objective in this section is also deliberately different from the cumulative
exposure objective of Section~\ref{sec:full_relaxation}.  We minimise the peak
reservoir level.  This is the appropriate safety benchmark because the threshold
condition is purely pointwise: if every post-release level is at most
\(\Delta_c\), then the reservoir remains below \(\Delta_c\) between releases as
well, since it only decays there.  Conversely, if a post-release level exceeds
\(\Delta_c\), the threshold is crossed immediately.  Thus the peak problem gives
an exact test for threshold safety in the scalar envelope, while the cumulative
finite-recovery exposure problem would be a different convex-control problem.

\subsection{Fixed inter-release time}
\label{subsec:finite_recovery_peak_problem}

Assume first that consecutive releases are separated by a fixed time interval
\(\tau>0\).  During this interval the reservoir decays by the factor
\begin{equation}\label{eq:finite_lambda_recall}
  \lambda=e^{-\rho\tau}\in(0,1).
\end{equation}
Starting from an empty reservoir, the post-release levels satisfy
\begin{equation}\label{eq:A_finite_recovery_levels_empty}
  A_1=q_1,
  \qquad
  A_k=\lambda A_{k-1}+q_k,
  \quad k=2,\ldots,n,
\end{equation}
where
\begin{equation}\label{eq:finite_recovery_budget}
  q_k\ge0,
  \qquad
  \sum_{k=1}^n q_k=Q.
\end{equation}
Define
\begin{equation}\label{eq:c_m_lambda_def}
  c_n(\lambda):=1+(n-1)(1-\lambda),
  \qquad n=1,2,\ldots .
\end{equation}
The peak-safety value is
\begin{equation}\label{eq:H_sequence_equivalent}
  H_n(0,Q)
  :=
  \min_{\substack{q_k\ge0\\ \sum_{k=1}^n q_k=Q}}
  \max_{1\le k\le n} A_k .
\end{equation}
The factor \(c_n(\lambda)\) is the effective capacity multiplier: one unit of
allowed peak can be filled initially, and each later stage can replenish only
the fraction \(1-\lambda\) that has dissipated since the previous release.

\begin{theorem}[Finite-recovery peak capacity]
\label{thm:finite_recovery_bellman_solution}
For every \(n\ge1\) and \(Q\ge0\),
\begin{equation}\label{eq:H_empty_closed_form}
  H_n(0,Q)
  =
  \frac{Q}{1+(n-1)(1-\lambda)}
  =\frac{Q}{c_n(\lambda)}.
\end{equation}
If \(Q>0\), the optimal post-release levels are uniquely determined and are all
equal:
\begin{equation}\label{eq:A_constant_optimal_profile}
  A_1^*=A_2^*=\cdots=A_n^*=H_n(0,Q).
\end{equation}
The corresponding release profile is
\begin{equation}\label{eq:q_front_loaded_profile}
  q_1^*=H_n(0,Q),
  \qquad
  q_2^*=\cdots=q_n^*=(1-\lambda)H_n(0,Q),
\end{equation}
and hence
\begin{equation}\label{eq:front_loading_ratio}
  \frac{q_1^*}{q_k^*}=\frac{1}{1-\lambda},
  \qquad k=2,\ldots,n.
\end{equation}
\end{theorem}

\begin{proof}
For any feasible sequence, the releases can be reconstructed from the
post-release levels by
\[
  q_1=A_1,
  \qquad
  q_k=A_k-\lambda A_{k-1},\quad k\ge2.
\]
Summing gives the capacity identity
\begin{equation}\label{eq:Q_capacity_identity}
  Q=A_n+(1-\lambda)\sum_{k=1}^{n-1}A_k .
\end{equation}
If all post-release levels are bounded by \(M\), then
\[
  Q\le M+(n-1)(1-\lambda)M=c_n(\lambda)M.
\]
Thus every feasible peak satisfies \(M\ge Q/c_n(\lambda)\).

The bound is attained by taking all post-release levels equal to
\(H:=Q/c_n(\lambda)\).  Then \(q_1=H\) and
\(q_k=(1-\lambda)H\) for \(k\ge2\), and these releases sum to \(Q\) by
\eqref{eq:Q_capacity_identity}.  If \(Q>0\), equality in the capacity bound
forces equality in every estimate \(A_k\le M\), so all optimal post-release
levels must be equal to \(H\).  The release profile then follows uniquely from
the reconstruction formula.
\end{proof}

\begin{remark}[State-dependent form]
The same calculation has a dynamic programming representation.  If \(a\ge0\)
is the level immediately before the first of \(m\) remaining releases, define
\begin{equation}\label{eq:H_bellman_state_value}
  H_m(a,Q)
  :=
  \min_{\substack{q_k\ge0\\ \sum_{k=1}^m q_k=Q}}
  \max_{1\le k\le m} A_k,
  \qquad
  A_1=a+q_1,
  \quad
  A_k=\lambda A_{k-1}+q_k .
\end{equation}
Then the same capacity identity gives
\begin{equation}\label{eq:H_closed_form}
  H_m(a,Q)
  =
  \max\left\{
      a,
      \frac{a+Q}{1+(m-1)(1-\lambda)}
  \right\}.
\end{equation}
Equivalently, this value satisfies the recursion
\begin{equation}\label{eq:H_bellman_reader_friendly}
  H_m(a,Q)
  =
  \min_{0\le q\le Q}
  \max\left\{
      a+q,
      H_{m-1}\bigl(\lambda(a+q),Q-q\bigr)
  \right\},
  \qquad H_1(a,Q)=a+Q.
\end{equation}
This recursion is the usual deterministic minimax representation of the same
capacity identity.
\end{remark}

\begin{remark}[Loading/maintenance form]
The profile \eqref{eq:q_front_loaded_profile} is the loading-plus-maintenance
structure of a linear carry-over model: the first release fills the empty
reservoir to the target level, and every later release replaces the dissipated
fraction \cite{RowlandTozer2011,Bauer2014}.  In the present paper the target
peak is compared with the stability-derived threshold \(\Delta_c\), and the
scalar envelope is justified by Proposition~\ref{prop:conservative_envelope}.
\end{remark}

\subsection{Threshold safety under finite recovery}
\label{subsec:finite_recovery_safety}

Starting from an empty reservoir, threshold safety is equivalent to
\(H_n(0,Q)\le\Delta_c\).  The previous theorem therefore gives the safe-capacity
criterion immediately.

\begin{corollary}[Finite-recovery threshold safety]
\label{cor:finite_recovery_safety}
Let \(r:=Q/\Delta_c\).  With \(n\) releases separated by a fixed recovery
interval \(\tau\) and with \(\lambda=e^{-\rho\tau}\), a total load \(Q\) is
threshold-safe from the empty reservoir if and only if
\begin{equation}\label{eq:finite_recovery_safe_condition}
  r\le c_n(\lambda)=1+(n-1)(1-\lambda).
\end{equation}
Equivalently,
\begin{equation}\label{eq:Qmax_fixed_lambda}
  Q\le Q_{\max}^{\rm safe}(n,\lambda)
  :=\Delta_c\bigl[1+(n-1)(1-\lambda)\bigr].
\end{equation}
The minimal number of releases required for threshold safety is
\begin{equation}\label{eq:N_safe_lambda}
  N_{\rm safe}^{\lambda}(Q)
  =
  1+
  \left\lceil
      \frac{(r-1)_+}{1-\lambda}
  \right\rceil,
  \qquad
  (x)_+:=\max\{x,0\}.
\end{equation}
\end{corollary}

\begin{proof}
The condition \(H_n(0,Q)\le\Delta_c\), together with
\eqref{eq:H_empty_closed_form}, is exactly
\(Q/\Delta_c\le c_n(\lambda)\).  Solving this inequality for the smallest
integer \(n\ge1\) gives \eqref{eq:N_safe_lambda}.
\end{proof}

When \(\lambda=0\), the reservoir fully recovers between releases and
\eqref{eq:q_front_loaded_profile} reduces to the equal split \(q_k=Q/n\).  When
\(\lambda\) is close to one, little recovery occurs between stages, and the
later releases can only be small maintenance increments.  Thus finite recovery
changes the natural safety allocation from equal splitting to front-loaded
constant-peak dispatch.

\subsection{A fixed crisis horizon}
\label{subsec:fixed_crisis_horizon}

The preceding formulas fix the inter-release interval \(\tau\).  In that
convention, increasing \(n\) also lengthens the total absorption process.  A
different question arises when the whole load must be processed within a fixed
horizon \(T>0\).  Then more releases mean shorter intervals and less recovery
between neighbouring stages.

For \(n\ge2\) equally spaced releases over a horizon \(T\), set
\begin{equation}\label{eq:lambda_n_fixed_T_again}
  \lambda_n=e^{-\rho T/(n-1)}.
\end{equation}
Writing \(x=\rho T\), define
\begin{equation}\label{eq:c_n_fixed_T_def}
  B_n(x)
  :=1+(n-1)\left(1-e^{-x/(n-1)}\right),
  \qquad n\ge2,
\end{equation}
with \(B_1(x):=1\).  Then
\begin{equation}\label{eq:H_fixed_T}
  H_n^{T}(0,Q)=\frac{Q}{B_n(\rho T)},
  \qquad n\ge2,
\end{equation}
and the threshold-safe condition is
\begin{equation}\label{eq:fixed_T_safe_condition}
  \frac{Q}{\Delta_c}\le B_n(\rho T).
\end{equation}

\begin{proposition}[Horizon-limited safe capacity]
\label{prop:horizon_limited_capacity}
Let \(x:=\rho T>0\).  The functions \(B_n(x)\) are increasing in \(n\) and
\begin{equation}\label{eq:c_n_fixed_T_limit}
  \lim_{n\to\infty}B_n(x)=1+x.
\end{equation}
Moreover, among all \(n\)-release schedules whose inter-release intervals fit
inside the same horizon \(T\), equal spacing maximises the safe capacity.  Over
all finite schedules, the supremal threshold-safe load is
\begin{equation}\label{eq:Qsup_safe_T}
  Q_{\sup}^{\rm safe}(T)=\Delta_c(1+\rho T).
\end{equation}
Consequently, if
\begin{equation}\label{eq:fixed_T_infeasible_condition}
  Q>\Delta_c(1+\rho T),
\end{equation}
then threshold safety is impossible within the horizon \(T\) in the scalar
envelope.  If \(Q<\Delta_c(1+\rho T)\), sufficiently fine equal spacing is
threshold-safe.  At equality the value is only supremal: for every finite
\(n\ge2\), \(B_n(\rho T)<1+\rho T\).
\end{proposition}

\begin{proof}
Write \(m=n-1\) and
\[
  f(m):=m\left(1-e^{-x/m}\right),
  \qquad m>0.
\]
Then \(B_n(x)=1+f(n-1)\).  Since
\[
  f'(m)=1-e^{-x/m}\left(1+\frac{x}{m}\right)>0,
\]
by \(e^y>1+y\) for \(y>0\), the sequence \(B_n(x)\) is increasing.  The limit
follows from \(1-e^{-x/m}=x/m+O(m^{-2})\).  This expansion also gives
\(B_n(x)<1+x\) for every finite \(n\ge2\).

For unequal spacings, let \(\tau_1,\ldots,\tau_{n-1}\ge0\) be the
inter-release intervals, with \(\sum_j\tau_j\le T\).  The capacity identity
becomes
\[
  Q\le M\left[
  1+\sum_{j=1}^{n-1}\bigl(1-e^{-\rho\tau_j}\bigr)
  \right]
\]
for any schedule whose peak is at most \(M\).  For fixed \(n\) and
\(\sum_j\tau_j=T\), concavity of \(\tau\mapsto1-e^{-\rho\tau}\) gives the
maximum at equal spacing.  For arbitrary \(n\) and arbitrary spacings within the
horizon,
\[
  1+\sum_{j=1}^{n-1}\bigl(1-e^{-\rho\tau_j}\bigr)
  \le
  1+\rho\sum_{j=1}^{n-1}\tau_j
  \le 1+\rho T.
\]
Taking \(M=\Delta_c\) gives the upper bound \(\Delta_c(1+\rho T)\).  Conversely,
because \(B_n(x)\uparrow 1+x\), every load strictly below this bound is safe for
all sufficiently large equally spaced \(n\).  Since \(1-e^{-y}<y\) for \(y>0\),
no finite schedule attains the bound exactly when \(T>0\).  This proves the
supremal capacity and the stated feasibility alternatives.
\end{proof}

With unlimited calendar time, additional releases can eventually make any fixed
load safe in the scalar envelope.  With a fixed horizon, the total amount of
available recovery is bounded by \(\rho T\).  The scalar envelope can therefore
absorb at most one threshold unit initially and at most \(\rho T\) further
threshold units through recovery over the horizon.  This is the origin of the
frontier \(Q=\Delta_c(1+\rho T)\).

% =================================================================
% Phase structure of the absorption benchmarks
% =================================================================

\section{Phase structure of the absorption benchmarks}
\label{sec:phase_diagram}

The preceding sections give explicit scalar-envelope benchmarks for three
effects: finite-horizon capacity, fixed-overhead stage choice, and
finite-recovery temporal allocation.  This section collects them in the
parameters natural for the compartmental threshold model.  All boundaries are
expressed in units of the stability-derived threshold \(\Delta_c\), and safety
in the scalar envelope is a sufficient condition for local threshold safety in
the full symmetric compartmental mechanism.

We use the dimensionless variables
\begin{equation*}
  r:=\frac{Q}{\Delta_c},
  \qquad
  h:=\rho T,
  \qquad
  k:=\frac{K\rho}{\mu-\beta}.
\end{equation*}
Here \(r\) is the load measured in threshold units, \(h\) is the available
recovery over the horizon, and \(k\) is the per-release overhead in the natural
one-shock exposure scale.  Figure~\ref{fig:phase_diagram} shows three exact
slices of this scalar structure.  Panel~(a) shows fixed-horizon feasibility:
the curves \(r=B_n(h)\) mark the capacity of \(n\) equally spaced releases, and
\(r=1+h\) is the limiting frontier.  Panel~(b) shows the complete-relaxation
setup-cost slice: \(k=k_{\rm safe}(r)\) separates the full-safe optimum from an
optimum with positive residual exposure.  Panel~(c) shows a finite-recovery
trajectory at \(n=3\), \(h=2\), \(\lambda=e^{-1}\), and \(r=2.1\): the uniform
split crosses the threshold, while the constant-peak allocation remains below
it.  The panels are complementary benchmarks, not one combined model in which
all effects are optimised simultaneously.

\begin{figure}[t]
  \centering
  \includegraphics[width=\textwidth]{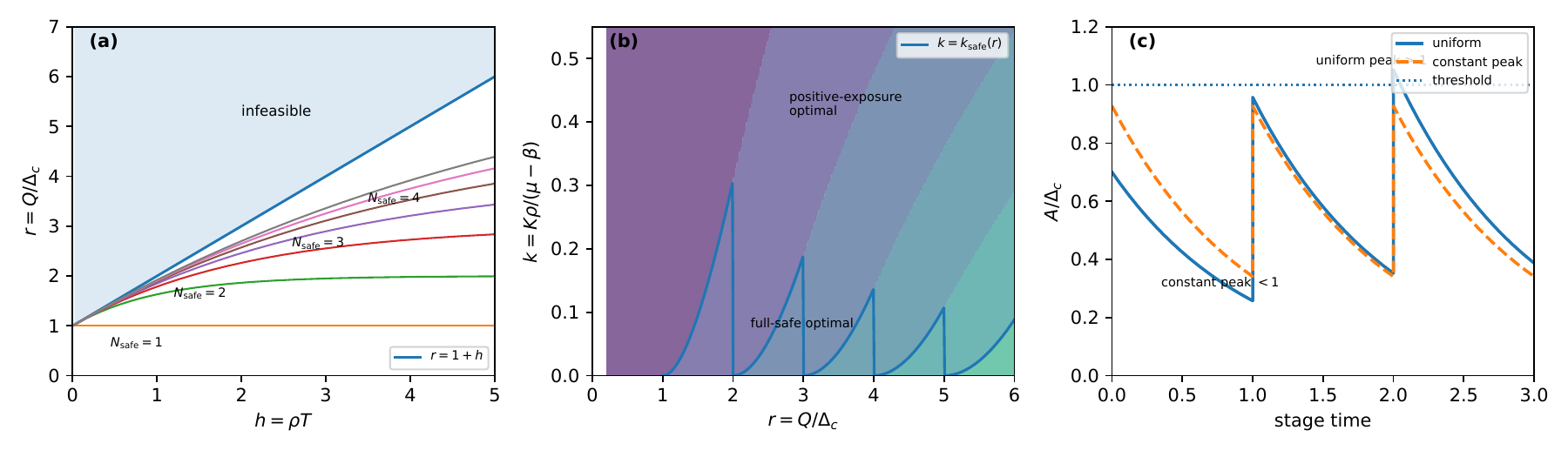}
  \caption{Three scalar threshold-safety benchmarks; in panel~(b), darker shading
above \(k=k_{\rm safe}(r)\) indicates regions with larger cost-optimal release
count \(n^*\)}
  \label{fig:phase_diagram}
\end{figure}

\subsection{Finite-horizon feasibility}
\label{ssec:phase_finite_horizon}

Panel~(a) is the pure capacity slice.  For \(n\ge2\) equally spaced releases
over a horizon \(T\), the inter-release carry-over factor is
\(e^{-\rho T/(n-1)}\).  Proposition~\ref{prop:horizon_limited_capacity} gives
the dimensionless safe capacity
\begin{equation}\label{eq:Bn_phase}
  B_n(h)
  =1+(n-1)\bigl(1-e^{-h/(n-1)}\bigr),
  \qquad h=\rho T,
\end{equation}
with \(B_1(h)=1\).  Equal spacing is not merely a drawing convention: for a
fixed \(n\) and \(T\), it maximises the scalar-envelope capacity among all
spacings.

Since \(B_n(h)\uparrow 1+h\) and \(B_n(h)<1+h\) for every finite
\(n\ge2\), the finite safe release count is
\begin{equation}\label{eq:Nsafefinite_phase}
  N_{\rm safe}(r,h)=
  \begin{cases}
    1,
      &0<r\le1,\\[0.4em]
    \min\{n\ge2:\ r\le B_n(h)\},
      &1<r<1+h,\\[0.4em]
    +\infty,
      &r\ge 1+h.
  \end{cases}
\end{equation}
Thus there are three regimes.  If \(1<r<1+h\), sufficiently fine finite
splitting is threshold-safe.  If \(r=1+h\), the boundary is only supremal: no
finite number of releases attains threshold safety, although the capacities
approach the boundary as \(n\to\infty\).  If
\begin{equation}\label{eq:finite_horizon_infeasible_phase}
  r>1+h,
\end{equation}
threshold safety is impossible even at the supremal scalar-envelope capacity.

\subsection{Cost-driven partial safety at complete relaxation}
\label{ssec:phase_cost}

Panel~(b) shows a different slice: complete relaxation between releases, but a
fixed overhead per release.  This is the standard setup-cost mechanism attached
to the complete-relaxation exposure benchmark.  Its role here is to show that
full threshold safety may be feasible and still fail to be cost-optimal.

For a fixed release count \(n\), the dimensionless objective is
\begin{equation}\label{eq:Jhatn_phase}
  \widehat{\mathcal J}_n(r,k)=nk+L_n(r),
\end{equation}
where
\begin{equation}\label{eq:Ln_phase}
  L_n(r)=
  \begin{cases}
    0, &0\le r\le n,\\[0.4em]
    r-n-n\log(r/n), &r>n.
  \end{cases}
\end{equation}
Let \(N=\lceil r\rceil\).  Since \(L_n(r)=0\) for all \(n\ge N\), the integer
minimisation reduces to \(n=1,\ldots,N\).  For \(r>1\), the fully safe candidate
\(n=N\) is optimal exactly when
\begin{equation}\label{eq:ksafe_phase}
  k\le k_{\rm safe}(r)
  :=
  \min_{1\le m<N}
  \frac{r-m-m\log(r/m)}{N-m}.
\end{equation}
For \(r\le1\), one release is already threshold-safe, so no boundary is needed.
Below \(k_{\rm safe}(r)\), the cost-optimal complete-relaxation policy is fully
safe.  Above it, the optimum uses fewer releases and accepts positive residual
threshold exposure.

The sawtooth form of \(k_{\rm safe}\) comes from the interaction of the integer
safe count \(\lceil r\rceil\) with the quadratic onset of the threshold exposure
just above each integer load level.  Just after \(r\) crosses an integer, the extra exposure removed by one
additional release is small, so even a small overhead can make partial safety
cost-optimal.

\subsection{A finite-recovery trajectory: why allocation matters}
\label{ssec:phase_front_loading}

Panel~(c) returns to finite recovery and illustrates a point that is hidden in
the complete-relaxation slice.  The capacity condition may say that a given
\((n,h,r)\) is safe in principle, but the release allocation still matters.  For
\begin{equation*}
  n=3,
  \qquad
  h=2,
  \qquad
  \lambda=e^{-1},
  \qquad
  r=2.1,
\end{equation*}
one has
\begin{equation}\label{eq:B3_phase}
  B_3(2)=1+2(1-e^{-1})\approx2.264,
  \qquad
  \frac{H_3}{\Delta_c}=\frac{r}{B_3(2)}\approx0.928.
\end{equation}
Thus three releases can keep the scalar reservoir below the threshold, but the
uniform split does not.  With \(q_j=Q/3\), the normalised post-release levels are
\begin{equation*}
  0.700,
  \qquad
  (1+e^{-1})0.700\approx0.958,
  \qquad
  (1+e^{-1}+e^{-2})0.700\approx1.052,
\end{equation*}
so the third stage crosses the threshold.

The peak-optimal allocation is the linear carry-over loading/maintenance
profile
\begin{equation}\label{eq:front_loaded_phase}
  q_1=H_3,
  \qquad
  q_2=q_3=(1-\lambda)H_3,
  \qquad
  H_3=\frac{Q}{B_3(2)}.
\end{equation}
It keeps all three post-release levels equal to \(H_3<\Delta_c\).  The conclusion is that, when the target level is the stability-derived
threshold \(\Delta_c\), an equal split can cross the threshold even though a
threshold-safe allocation with the same \(Q\), \(n\), and \(T\) exists.

\subsection{What the phase diagram shows}
\label{ssec:phase_synthesis}

The figure separates three obstructions.  The first is the threshold barrier:
when \(r>1\), a single release is supercritical.  The second is the calendar
capacity barrier: when \(r>1+h\), the scalar envelope cannot absorb the load
within the horizon, no matter how finely it is split.  The third is the
overhead barrier: when \(k>k_{\rm safe}(r)\) in the complete-relaxation
benchmark, full safety remains feasible but is not cost-optimal.

These statements are deliberately scalar-envelope statements.  By
Proposition~\ref{prop:conservative_envelope}, a schedule that is threshold-safe
in the scalar envelope is also locally threshold-safe for the full symmetric
model.  The converse is not claimed: crossing a scalar-envelope boundary means
that this conservative sufficient condition fails, not that mobilisation is
unavoidable in the full nonlinear system.

% =================================================================
% Discussion
% =================================================================

\section{Discussion}\label{sec:discussion}

The paper is positioned inside compartmental threshold models of political
shock absorption.  The original step is the bridge from a local voter-flow
threshold mechanism to a scalar release-scheduling benchmark.  The threshold
\(\Delta_c\) is obtained from local stability, the exposure functional is tied
to positive logarithmic amplification of the mobilisation coordinate, and the
scalar reservoir is justified by comparison as a conservative envelope for the
nonlinear two-variable system.  Hence a scalar threshold-safe schedule is a
sufficient safety certificate for the full symmetric mechanism.

This bridge gives the model its operations-research interpretation.  After the
reduction, the problem becomes a dispatching and release-scheduling problem
under a threshold constraint: how many stages to use, how much load to allocate
to each stage, how recovery changes capacity, and how fixed setup overhead
changes the cost-optimal stage count.  The complete-relaxation model is a
convex allocation problem with a zero-loss buffer; the overhead model is a
setup-cost trade-off; the fixed-horizon model is a leaky-reservoir capacity
problem.  These are standard OR structures, but here their capacity and loss
levels are inherited from the stability analysis of the compartmental model.

The scalar layer has close analogues in several applied fields.  Table~\ref{tab:scalar_analogues}
summarises the shared skeleton and the distinction used in this paper.

\begin{table}[t]
\centering
\small
\caption{Relation to scalar analogues}
\label{tab:scalar_analogues}
\begin{tabular}{@{}p{0.19\textwidth}p{0.34\textwidth}p{0.38\textwidth}@{}}
\hline
\textbf{Field} & \textbf{Shared scalar skeleton} & \textbf{Difference here}\\
\hline
Pharmacokinetics & First-order elimination; loading/maintenance dosing & Target threshold comes from voter-flow stability.\\[2pt]
Radiotherapy & Fractionation with repair between stages & Exposure is positive log-growth pressure, not BED.\\[2pt]
Reservoir operation & Leaky storage under capacity constraints & Capacity is a conservative envelope of an ODE reservoir.\\[2pt]
Setup-cost scheduling & Fixed overhead per stage & Stage loss is threshold exposure.\\[2pt]
Carbon budgets & Allocation of a remaining load under a threshold & Policy-level analogy, not identical dynamics.\\
\hline
\end{tabular}
\end{table}

The table also delimits the mathematical claim.  The scalar recurrence
\(A_k=\lambda A_{k-1}+q_k\), Jensen allocation, constant-peak dispatch, and
setup-cost comparison are familiar mechanisms.  Their role here is to make the
threshold regimes transparent after a stability-derived reservoir coordinate
has been extracted from the ODE model.

The methodological motivation came from shock splitting in supersonic gas
dynamics \cite{CourantFriedrichs1976,LiepmannRoshko1957,Anderson2003,
HendersonMenikoff1998} and from earlier work on optimal shock-wave systems
\cite{MO1998,MO2003a,MO2006}.  The formal mechanism is different.  Gas-dynamic
shocks are discontinuities in weak solutions of hyperbolic conservation laws,
with Rankine--Hugoniot relations and entropy admissibility.  The present impulse
update is an externally imposed control in an ODE reservoir.  The replacement is
the ODE balance structure: \(\Phi\) dissipates along the autonomous flow,
\(g(A)\) is the near-baseline logarithmic growth pressure, and the positive-part
exposure bounds positive logarithmic amplification.  This is the ODE-level
analogue of deriving the loss from the governing dynamics rather than imposing
an external quadratic penalty.

The interpretation remains deliberately narrow.  The releases \(q_j\) are not
actions that directly move voters between parties or ideological states.  They
represent effective portions of an external load after administrative, fiscal,
communicative, or institutional buffering.  The model is not a forecasting
model, and the quantities \(Q\), \(\Delta_c\), \(\rho\), and \(K\) are parameters
of a stylised reduction.  The scalar boundaries in the phase diagram mark
sufficient safety certificates and conservative capacity limits.  If the scalar
envelope crosses the threshold, the full nonlinear system may still be moderated
by reservoir depletion and mobilisation saturation; the converse implication is
not claimed.

Several extensions are natural.  The first is the finite-recovery
cumulative-exposure problem: one can minimise
\[
  \sum_{k=1}^{n-1}\int_0^\tau [g(A_k e^{-\rho s})]_+\,ds
  +\int_0^\infty [g(A_n e^{-\rho s})]_+\,ds,
  \qquad A_k=\lambda A_{k-1}+q_k,
\]
over \(q_k\ge0\) with \(\sum_k q_k=Q\).  This remains a convex problem for the
linear envelope, but it is not the same as the constant-peak capacity benchmark.
Other extensions include stochastic loads and chance constraints, state-
dependent recovery, heterogeneous or network reservoirs with Perron--Frobenius
safe sets, and joint optimisation of release times and release sizes.

In summary, the paper introduces a conservative threshold-safety layer for a
specific class of compartmental voter-flow models.  The scalar formulae are
closely related to classical leaky-reservoir models; their role is to make
explicit how a stability-derived threshold, finite recovery, fixed horizon, and
per-stage overhead interact in this compartmental setting.

%% =============================================================
%% Statements and Declarations
%% =============================================================

\section*{Statements and Declarations}

\noindent\textbf{Competing interests.} The author declares no competing interests.

\medskip
\noindent\textbf{Funding.} No funding was received for this work.

\medskip
\noindent\textbf{Data availability.} No datasets were generated or analysed during the current study.

\medskip
\noindent\textbf{Code availability.} The code used to generate the illustrative figures is available from the author upon reasonable request.

%% =============================================================
%% References 
%% =============================================================

\bibliographystyle{unsrt}
\bibliography{shock_absorption_references}

@article{McCluskeySantoprete2018,
  author  = {McCluskey, C. Connell and Santoprete, Manuele},
  title   = {A Bare-bones Mathematical Model of Radicalization},
  journal = {Journal of Dynamics and Games},
  volume  = {5},
  number  = {3},
  pages   = {243--264},
  year    = {2018},
  doi     = {10.3934/jdg.2018016}
}

@article{SantopreteXu2018,
  author  = {Santoprete, Manuele and Xu, Fei},
  title   = {Global Stability in a Mathematical Model of De-radicalization},
  journal = {Physica A: Statistical Mechanics and its Applications},
  volume  = {509},
  pages   = {151--161},
  year    = {2018},
  doi     = {10.1016/j.physa.2018.06.027}
}

@article{Santoprete2019,
  author  = {Santoprete, Manuele},
  title   = {Countering Violent Extremism: A Mathematical Model},
  journal = {Applied Mathematics and Computation},
  volume  = {358},
  pages   = {314--329},
  year    = {2019},
  doi     = {10.1016/j.amc.2019.04.054}
}

@book{CourantFriedrichs1976,
  author    = {Courant, Richard and Friedrichs, Kurt Otto},
  title     = {Supersonic Flow and Shock Waves},
  series    = {Applied Mathematical Sciences},
  publisher = {Springer},
  address   = {New York},
  year      = {1976},
  doi       = {10.1007/978-1-4684-9364-1}
}

@book{LiepmannRoshko1957,
  author    = {Liepmann, Hans Wolfgang and Roshko, Anatol},
  title     = {Elements of Gasdynamics},
  publisher = {Wiley},
  address   = {New York},
  year      = {1957}
}

@book{Anderson2003,
  author    = {Anderson, John D.},
  title     = {Modern Compressible Flow: With Historical Perspective},
  edition   = {3},
  publisher = {McGraw-Hill},
  address   = {New York},
  year      = {2003}
}

@article{HendersonMenikoff1998,
  author  = {Henderson, Le Roy F. and Menikoff, Ralph},
  title   = {Triple-shock Entropy Theorem and Its Consequences},
  journal = {Journal of Fluid Mechanics},
  volume  = {366},
  pages   = {179--210},
  year    = {1998},
  doi     = {10.1017/S0022112098001244}
}

@article{MO1998,
  author  = {Malozemov, V. N. and Omelchenko, A. V. and Uskov, V. N.},
  title   = {On the Minimization of Total-pressure Losses in Decelerating a Supersonic Flow},
  journal = {Prikladnaya Matematika i Mekhanika},
  volume  = {62},
  number  = {6},
  pages   = {1014--1020},
  year    = {1998},
  note    = {English translation: Journal of Applied Mathematics and Mechanics, 62(6), 939--944}
}

@article{MO2003a,
  author  = {Malozemov, V. N. and Omelchenko, A. V.},
  title   = {On the Construction of Optimal Shock Wave Systems},
  journal = {Computational Mathematics and Mathematical Physics},
  volume  = {43},
  number  = {4},
  pages   = {508--520},
  year    = {2003},
  note    = {Russian original: Zhurnal Vychislitel'noi Matematiki i Matematicheskoi Fiziki, 43(4), 533--545}
}

@article{MO2006,
  author  = {Malozemov, V. N. and Omelchenko, A. V.},
  title   = {On a Discrete Optimal Control Problem with an Explicit Solution},
  journal = {Journal of Industrial and Management Optimization},
  volume  = {2},
  number  = {1},
  pages   = {55--62},
  year    = {2006},
  doi     = {10.3934/jimo.2006.2.55}
}

@book{RowlandTozer2011,
  author    = {Rowland, Malcolm and Tozer, Thomas N.},
  title     = {Clinical Pharmacokinetics and Pharmacodynamics: Concepts and Applications},
  edition   = {4},
  publisher = {Wolters Kluwer Health/Lippincott Williams \& Wilkins},
  address   = {Philadelphia},
  year      = {2011},
  isbn      = {9780781750097}
}

@book{Bauer2014,
  author    = {Bauer, Larry A.},
  title     = {Applied Clinical Pharmacokinetics},
  edition   = {3},
  publisher = {McGraw-Hill Medical},
  address   = {New York},
  year      = {2014},
  isbn      = {9780071794589}
}

@article{WithersThamesPeters1983,
  author  = {Withers, H. Rodney and Thames, Howard D. and Peters, Lester J.},
  title   = {A New Isoeffect Curve for Change in Dose per Fraction},
  journal = {Radiotherapy and Oncology},
  volume  = {1},
  number  = {2},
  pages   = {187--191},
  year    = {1983},
  doi     = {10.1016/S0167-8140(83)80021-8}
}

@article{Fowler1989,
  author  = {Fowler, John F.},
  title   = {The Linear-quadratic Formula and Progress in Fractionated Radiotherapy},
  journal = {British Journal of Radiology},
  volume  = {62},
  number  = {740},
  pages   = {679--694},
  year    = {1989},
  doi     = {10.1259/0007-1285-62-740-679}
}

@article{Yeh1985,
  author  = {Yeh, William W.-G.},
  title   = {Reservoir Management and Operations Models: A State-of-the-Art Review},
  journal = {Water Resources Research},
  volume  = {21},
  number  = {12},
  pages   = {1797--1818},
  year    = {1985},
  doi     = {10.1029/WR021i012p01797}
}

@article{WagnerWhitin1958,
  author  = {Wagner, Harvey M. and Whitin, Thomson M.},
  title   = {Dynamic Version of the Economic Lot Size Model},
  journal = {Management Science},
  volume  = {5},
  number  = {1},
  pages   = {89--96},
  year    = {1958},
  doi     = {10.1287/mnsc.5.1.89}
}

@article{Allen2009,
  author  = {Allen, Myles R. and Frame, David J. and Huntingford, Chris and Jones, Chris D. and Lowe, Jason A. and Meinshausen, Malte and Meinshausen, Nicolai},
  title   = {Warming Caused by Cumulative Carbon Emissions towards the Trillionth Tonne},
  journal = {Nature},
  volume  = {458},
  number  = {7242},
  pages   = {1163--1166},
  year    = {2009},
  doi     = {10.1038/nature08019}
}

@article{Meinshausen2009,
  author  = {Meinshausen, Malte and Meinshausen, Nicolai and Hare, William and Raper, Sarah C. B. and Frieler, Katja and Knutti, Reto and Frame, David J. and Allen, Myles R.},
  title   = {Greenhouse-gas Emission Targets for Limiting Global Warming to 2 Degrees C},
  journal = {Nature},
  volume  = {458},
  number  = {7242},
  pages   = {1158--1162},
  year    = {2009},
  doi     = {10.1038/nature08017}
}

@misc{Omelchenko2026Threshold,
  author        = {Omelchenko, Alexander},
  title         = {Threshold Dynamics of Voter Radicalization on the Probability Simplex},
  year          = {2026},
  archivePrefix = {arXiv},
  eprint        = {2603.07862},
  primaryClass  = {math.DS},
  note          = {arXiv:2603.07862v1}
}

@article{AzizahBakhtiarSianturi2023,
  author    = {Azizah, Nur and Bakhtiar, Toni and Sianturi, Paian},
  title     = {Modeling and Control of the Extreme Ideology Transmission
               Dynamics in a Society},
  journal   = {Jambura Journal of Mathematics},
  volume    = {5},
  number    = {1},
  year      = {2023},
  doi       = {10.34312/jjom.v5i1.15583},
  publisher = {Universitas Negeri Gorontalo}
}

@article{TapsobaSimporeTraore2025,
  author    = {Tapsoba, Wendpanga Alain and Simpor{\'e}, Yacouba and
               Traor{\'e}, Oumar},
  title     = {Mathematical Modeling and Optimal Control of the Dynamics of an
               Internal Radicalization Process},
  journal   = {Abstract and Applied Analysis},
  volume    = {2025},
  number    = {1},
  pages     = {9956552},
  year      = {2025},
  doi       = {10.1155/aaa/9956552},
  publisher = {John Wiley \& Sons}
}

\end{document}